\documentclass{amsart}
\usepackage{amssymb}
\usepackage{synttree}
\newtheorem{theorem}{Theorem}[section]
\newtheorem{prop}[theorem]{Proposition}
\newtheorem{lemma}[theorem]{Lemma}
\newtheorem{corollary}[theorem]{Corollary}

\newcounter{Examplecount}
\newcommand\beq{\begin{equation}}
\newcommand\eeq{\end{equation}}
\newcommand\bce{\begin{center}}
\newcommand\ece{\end{center}}
\newcommand\bea{\begin{eqnarray}}
\newcommand\eea{\end{eqnarray}}
\newcommand\bean{\begin{eqnarray*}}
\newcommand\eean{\end{eqnarray*}}
\newcommand\bmt{\begin{multline*}}
\newcommand\emt{\end{multline*}}
\newcommand\ben{\begin{enumerate}}
\newcommand\een{\end{enumerate}}
\newcommand\bit{\begin{itemize}}
\newcommand\eit{\end{itemize}}
\newcommand\brr{\begin{array}}
\newcommand\err{\end{array}}
\newcommand\bt{\begin{tabular}}
\newcommand\et{\end{tabular}}

\newcommand\nn{\nonumber}

\newcommand\ol{\overline}
\newcommand\wt{\widetilde}
\renewcommand\L{\Lambda}
\renewcommand\l{\lambda}
\newcommand\N{\mathbb{N}}
\newcommand\T{\mathcal{T}}
\newcommand\A{\mathcal{A}}
\newcommand\B{\mathcal{B}}
\newcommand\C{\mathcal{C}}
\newcommand\I{\mathcal{I}}
\newcommand\la{\langle}
\newcommand\ra{\rangle}

\author{Sergi Elizalde}
\title{Improved bounds on the number of numerical semigroups of a given genus}
\address{Department of Mathematics, Dartmouth College, Hanover, NH 03755-3551}

\begin{document}
\maketitle

\begin{abstract}
We improve the previously best known lower and upper bounds on the number $n_g$ of numerical semigroups of genus $g$.
 Starting from a known recursive description of the tree $\T$ of numerical semigroups, we analyze some of its properties and use them
to construct approximations of $\T$ by generating trees whose nodes are labeled by certain parameters
of the semigroups. We then translate the succession rules of these trees into functional equations
for the generating functions that enumerate their nodes, and solve these equations to obtain the bounds. Some of our bounds involve
the Fibonacci numbers, and the others are expressed as generating functions.

We also give upper bounds on the number of numerical semigroups having an infinite number of descendants in $\T$.
\end{abstract}

\section{Introduction}\label{sec:intro}

A numerical semigroup is a subset $\L$ of the non-negative integers $\mathbb{N}_0$ which contains $0$, is closed under addition, and such that $\mathbb{N}_0\setminus\L$ is finite.
The elements in $\mathbb{N}_0\setminus\L$ are called {\em gaps}, and the number of gaps is called the {\em genus} of $\L$, usually denoted by $g$. In this paper we are concerned with the number $n_g$ of numerical semigroups
of genus $g$. The sequence $n_g$ has been studied in~\cite{B,B2}. In~\cite{B}, Bras-Amor\'os gives the following bounds for $g\ge2$: \beq\label{eq:brasbounds} 2F_g\le n_g\le 1+3\cdot2^{g-3},\eeq
where $F_i$ is the Fibonacci sequence starting with $F_0=0$, $F_1=1$. The main purpose of this paper is to improve both these bounds.
In~\cite{B2} it is conjectured that \beq\label{eq:fibconj}\lim_{g\rightarrow\infty}\frac{n_{g+1}}{n_g}=\phi,\eeq
where $\phi=\frac{1+\sqrt{5}}{2}$ is the Golden ratio; in other words, the numbers $n_g$ grow exponentially at the same rate as the Fibonacci numbers.

The largest gap $f$ of $\L$ is called the {\em Frobenius number}.
The elements of $\L$ in increasing order are denoted by $0=\lambda_0<\lambda_1<\lambda_2<\dots$, and $\lambda_1$ is called the {\em multiplicity} of $\L$.

It is well known that $f<2g$. Indeed, if $2g\notin\L$, then $\L$ would contain at least $g$ of the numbers $\{1,2,\dots,2g-1\}$, so by the Pigeonhole principle, either $g\in\L$ or $\L$ would contain one of
the pairs $\{i,2g-i\}$, $1\le i\le g-1$, which would imply that $2g\in\L$.
A similar argument shows that $a\in\L$ for all $a\ge2g+1$, and that every $a\ge2g+2$ can be written as the sum of two nonzero elements of $\L$.

It is also well known that every numerical semigroup has a unique minimal (and finite) set of generators. If we denote by $\mu_1<\mu_2<\dots<\mu_m$ the minimal generators of $\L$,
the last sentence of the previous paragraph implies that $\mu_i\le 2g+1$ for all $i$. It is also clear that $\mu_1=\lambda_1$, and that $m\le\lambda_1$, since the minimal generators must be in different residue classes modulo $\lambda_1$.
Following the terminology from~\cite{BB}, we call $\mu_i$ an {\em effective generator} if $\mu_i>f$. If $\mu_{r+1},\mu_{r+2},\dots,\mu_m$ are the effective generators of $\L$, we write
$\L=\langle \mu_1,\dots,\mu_{r}|\mu_{r+1},\dots,\mu_m \rangle$. We denote by $e=e(\L)=m-r$ the number of effective generators.
An effective generator $\mu_j$ is said to be {\em strong}
if $\mu_1+\mu_j$ is a minimal generator of $\L\setminus\{\mu_j\}$, and it is called {\em weak} otherwise. Additionally, we say that an effective generator $\mu_j$ is {\em very weak}
if $\mu_1+\mu_j>2g+3$, and that it is {\em healthy} otherwise. 
Note that very weak generators are in particular weak, since any minimal generator of $\L\setminus\{\mu_j\}$ must be less than or equal to $2(g+1)+1$.
For example, $\L=\la 6,9|13,14,16,17 \ra$ has genus $g=9$ and $e=4$ effective generators, of which $13$ and $14$ are strong (and thus healthy), $16$ is weak but healthy, and $17$ is very weak (and thus weak).

A generating tree for all numerical semigroups is described in~\cite{B}, using a construction from~\cite{RGGJ}. The root of the tree is the semigroup $\N_0$, and for each numerical semigroup
$\L$ of genus $g\ge1$ and Frobenius number $f$, its parent is defined to be the numerical semigroup $\L\cup\{f\}$, which has genus $g-1$. The nodes at distance $g$ from the root correspond then to the $n_g$ numerical semigroups of genus $g$
(see Figure~\ref{fig:semigrouptree}).
It is easy to check that the children of a numerical semigroup in this tree are obtained by removing its effective generators one at a time.
In our notation, the children of $\L=\langle \mu_1,\dots,\mu_{r}|\mu_{r+1},\dots,\mu_{r+e} \rangle$ are $\L\setminus\{\mu_{r+i}\}$, $1\le i\le e$. It will be more convenient for us to consider the tree $\T$ that is obtained from this one
by removing the root $\N_0$. The root of $\T$ is then the semigroup $\{0,2,3,\dots\}=\langle |2,3\rangle$, of genus $1$.
The nodes at distance $g-1$ from the root (i.e., the numerical semigroups of genus $g$) are said to be at {\em level} $g$ in $\T$.

\begin{figure}[htb]
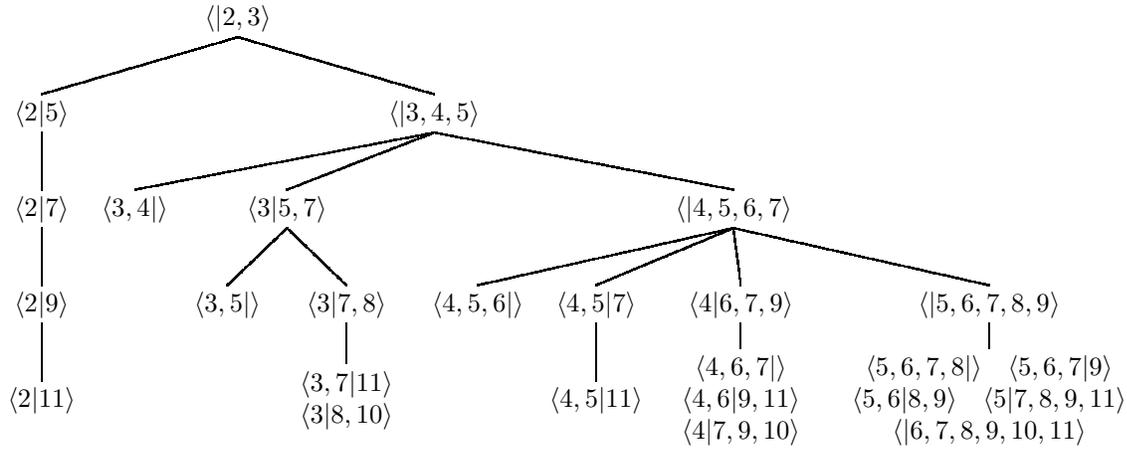

\begin{center}\synttree
[$\la|2,3\ra$ [$\la2|5\ra$ [ $\la2|7\ra$ [$\la2|9\ra$ [$\la2|11\ra$]]]]
[$\la|3,4,5\ra$ [$\la3,4|\ra$] [$\la3|5,7\ra$ [$\la3,5|\ra$] [$\la3|7,8\ra$ [\bt{c} $\la3,7|11\ra$ \\ $\la3|8,10\ra$ \et]] ]
[$\la|4,5,6,7\ra$ [$\la 4,5,6|\ra$] [$\la4,5|7\ra$ [$\la4,5|11\ra$]] [$\la4|6,7,9\ra$ [\bt{c} $\la4,6,7|\ra$ \\ $\la4,6|9,11\ra$ \\ $\la4|7,9,10\ra$ \et]]
[$\la|5,6,7,8,9\ra$ [\bt{c} $\la5,6,7,8|\ra$\quad $\la5,6,7|9\ra$ \\ $\la5,6|8,9\ra$\quad $\la5|7,8,9,11\ra$\\ $\la|6,7,8,9,10,11\ra$\et ]]]]]
\end{center}
\caption{The first five levels of the tree $\T$. \label{fig:semigrouptree}}
\end{figure}

In the rest of the paper, the word {\it semigroup} will always refer to {\it numerical semigroup}.
For each $g$, the semigroup $O_g=\{0,g+1,g+2,g+3,\dots\}=\langle |g+1,g+2,\dots,2g+1 \rangle$ is called {\em ordinary}. Note that $O_g$ has genus $g$ and $e(O_g)=g+1$.
Some facts about the children of nodes in $\T$ are studied in~\cite{B}.
\begin{lemma}[\cite{B}]\label{lem:Maria}\renewcommand{\labelenumi}{(\alph{enumi})}
\ben
\item For $g\ge1$, the children of the ordinary semigroup $O_g$ are $$\begin{cases} O_{g+1}=\langle |g+2,g+3,\dots,2g+2,2g+3 \rangle, \\
\langle g+1|g+3,g+4,\dots,2g+1,2g+3 \rangle, \mbox{ and}\\
\langle g+1,g+2,\dots,g+i-1|g+i+1,g+i+2,\dots,2g+1 \rangle \mbox{ for }3\le i\le g+1,\end{cases}$$
which have $g+2$, $g$, and $g-i+1$ (for $3\le i\le g+1$) effective generators, respectively.
\item Let $\L=\langle \mu_1,\dots,\mu_{r}|\mu_{r+1},\dots,\mu_{r+e} \rangle$ be a non-ordinary semigroup. For $1\le i\le e$, its child
$\L\setminus\{\mu_{r+i}\}$ is the semigroup $$\langle \mu_1,\dots,\mu_{r+i-1}|\mu_{r+i+1},\dots,\mu_{r+e} \rangle$$
(which has $e-i$ effective generators) if $\mu_{r+i}$ is a weak generator of $\L$, and it is the semigroup
$$\langle \mu_1,\dots,\mu_{r+i-1}|\mu_{r+i+1},\dots,\mu_{r+e},\mu_1+\mu_{r+i} \rangle$$
(which has $e-i+1$ effective generators) if $\mu_{r+i}$ is a strong generator of $\L$.
\een
\end{lemma}

In the last part of the above statement, the inequality $\mu_{r+e}<\mu_1+\mu_{r+i}$ follows from the fact that $\mu_{r+e}-\mu_1\notin\L$, so $\mu_{r+e}-\mu_1\le f<\mu_{r+i}$.
As a consequence of Lemma~\ref{lem:Maria}, if a non-ordinary semigroup has $e$ effective generators, then the numbers of effective generators of its children in $\T$ are $j_1,j_2\dots,j_e$, with each $j_i\in\{i-1,i\}$.
For the ordinary semigroup $O_{e-1}$ with $e$ effective generators (assume $e\ge2$), the numbers of effective generators of its children are $0,1,\dots,e-3,e-1,e+1$ (where $e+1$ corresponds to the child $O_e$).

The method used in~\cite{B} to derive the lower bound on $n_g$ from equation~(\ref{eq:brasbounds}) can be summarized as follows.
Consider the generating tree $\mathcal{A}$
with root $\ol{(2)}$ and succession rules (which recursively describe the children of each node) \bean\ol{(e)}&\longrightarrow&(0)(1)\dots(e-3)(e-1)\ol{(e+1)},\\ (e)&\longrightarrow&(0)(1)\dots(e-1).\eean
The tree $\mathcal{A}$ can be embedded in $\T$, that is, there is an injective map $\varphi$ from the nodes of $\A$ to the nodes of $\T$ fixing the root and such that if
$x$ is a child of $y$ in $\A$, then $\varphi(x)$ is a child of $\varphi(y)$ in $\T$. Such a map can be constructed recursively level by level so that each node $\ol{(e)}$ in $\A$
is mapped to the semigroup $O_{e-1}$ in $\T$, and each node $(e)$ is mapped to a non-ordinary semigroup with at least $e$ effective generators.
We use the notation $\mathcal{A}\prec\T$ to indicate that $\mathcal{A}$ can be embedded in $\T$. Now the lower bound follows by proving inductively that $\A$ has $2F_g$ nodes at level $g$.
Similarly, the upper bound from equation~(\ref{eq:brasbounds}) is derived in~\cite{B} by considering the tree $\mathcal{B}$
with root $\ol{(2)}$ and succession rules \bean\ol{(e)}&\longrightarrow&(0)(1)\dots(e-3)(e-1)\ol{(e+1)},\\ (e)&\longrightarrow&(1)(2)\dots(e),\eean which has $1+3\cdot2^{g-3}$ nodes at level $g$ and satisfies $\T\prec\mathcal{B}$.

In Section~\ref{sec:lower} we give lower bounds on the number $n_g$ of numerical semigroups of genus $g$. First we improve the known $2F_{g}$ bound from~\cite{B} to $F_{g+2}-1$, and then
we use a more sophisticated argument to further improve it. In Section~\ref{sec:upper} we give an improved upper bound on $n_g$, constructing a generating tree with unusual succession rules.
Finally, in Section~\ref{sec:infinite} we give two upper bounds on the number of numerical semigroups of genus $g$ with an infinite number of descendants in $\T$, one involving the Fibonacci numbers
and the other in terms of the numbers $n_g$.

We will be using generating functions in many of the proofs to enumerate the nods of generating trees. The variable $t$ will always mark the level of a node in the tree, which corresponds to the genus of a semigroup.
If $A(t)=\sum_{g\ge1} a_g t^g$, then $[t^g]A(t)=a_g$ denotes the coefficient of $t^g$ in $A(t)$.

\section{Improved lower bounds}\label{sec:lower}

\subsection{A simple bound}

For $g\ge1$ and $i\ge3$, let $$P_{g,i}=\langle g+1|g+i,g+i+1,\dots,\widehat{d(g+1)},\dots,2g+i \rangle,$$ where the hat indicates that $d(g+1)$ is missing, and $d$ is the unique integer such that
$g+i\le d(g+1)\le 2g+i$. Clearly $P_{g,i}$ has genus $g+i-2$ and $g$ effective generators. Part (a) of Lemma~\ref{lem:Maria} shows that $P_{g,3}$ is a child of $O_g$ in $\T$.
Also, removing the smallest effective generator of $P_{g,i}$ we get $P_{g,i+1}$, so $P_{g,i+1}$ is a child of $P_{g,i}$. In particular, the numbers of effective generators of the children of $P_{g,i}$ in $\T$ are
$j_1,j_2,\dots,j_{g-1},g$, where each $j_k\in\{k-1,k\}$, and $g$ corresponds to the child $P_{g,i+1}$.
This additional information can be used to improve the lower bound~(\ref{eq:brasbounds}) on the number of numerical semigroups of a given genus.

\begin{prop} For $g\ge1$, we have $n_g\ge F_{g+2}-1$.\label{prop:lowerbound1}
\end{prop}

\begin{proof}
We modify the generating tree $\A$ described in Section~\ref{sec:intro} by allowing special labels $\wt{(g)}$ for the semigroups $P_{g,i}$, so that the children of $\wt{(g)}$ are now labeled $(0)(1)\dots(g-2)\wt{(g)}$.
Let $\mathcal{A}'$ be the generating tree with root $\ol{(2)}$ and succession rules
\bea\ol{(e)}&\longrightarrow&(0)(1)\dots(e-3)\wt{(e-1)}\ol{(e+1)},\nn\\ \wt{(e)}&\longrightarrow&(0)(1)\dots(e-2)\wt{(e)},\label{eq:genrulea'}\\ (e)&\longrightarrow&(0)(1)\dots(e-1).\nn\eea
Clearly $\mathcal{A}'\prec\T$, so if $\ell_g$ is the number of nodes in $\mathcal{A}'$ at level $g$, we have $n_g\ge \ell_g$.

To find $\ell_g$, consider the generating functions $\ol{F}(u,t)$, $\wt{F}(u,t)$, and $F(u,t)$, where the coefficient of $u^et^g$ is the number of nodes in $\mathcal{A}'$ at level $g$ and label $\ol{(e)}$, $\wt{(e)}$, and $(e)$, respectively. Then $\ell_g$ is the coefficient of $t^g$ in $L(t):=\ol{F}(1,t)+\wt{F}(1,t)+F(1,t)$.

We have that $$\ol{F}(u,t)=u^2t+u^3t^2+\dots=\frac{u^2t}{1-ut},$$ since there is a node $\ol{g+1}$ at each level $g\ge1$, corresponding to the ordinary semigroup $O_g$.
To find an equation for $\wt{F}(u,t)$, note from the rules (\ref{eq:genrulea'}) that nodes $\wt{(e)}$ at level $g+1$ are children of nodes
$\wt{(e)}$ and $\ol{(e+1)}$ at level $g$, so $$\wt{F}(u,t)=t\left(\wt{F}(u,t)+\frac{1}{u}\ol{F}(u,t)\right),$$ from where $$\wt{F}(u,t)=\frac{ut^2}{(1-t)(1-ut)}.$$
Finally, to obtain an equation for $F(u,t)$, we see from the succession rules that each term $u^et^g$ of $\ol{F}(u,t)$, $\wt{F}(u,t)$, and $F(u,t)$ contributes to the coefficient of $t^{g+1}$ in $F(u,t)$ as
$1+u+\dots+u^{e-3}=(u^{e-2}-1)/(u-1)$, $(u^{e-1}-1)/(u-1)$, and $(u^{e}-1)/(u-1)$, respectively, so $$F(u,t)=\frac{t}{u-1}\left(\frac{1}{u^2}\ol{F}(u,t)-\ol{F}(1,t)+\frac{1}{u}\wt{F}(u,t)-\wt{F}(1,t)+F(u,t)-F(1,t)\right).$$
Substituting the known expressions for $\wt{F}(u,t)$ and $\ol{F}(u,t)$, we get $$\left(1-\frac{t}{u-1}\right)F(u,t)=\frac{t}{u-1}\left(-F(1,t)+\frac{t^2(u-1)}{(1-ut)(1-t)^2}\right).$$ The kernel of this equation is canceled by setting $u=1+t$,
from where we get $$F(1,t)=\frac{t^3}{(1-t-t^2)(1-t)^2} \quad\mbox{and}\quad L(t)=\frac{t}{(1-t-t^2)(1-t)}.$$ The series expansion of $L(t)$ gives the lower bound $\ell_g=F_{g+2}-1$.
\end{proof}

\subsection{A better bound}

We can further analyze the semigroups $P_{g,i}$ to obtain more information about their descendants in $\T$.

\begin{lemma}\label{lem:Pgi} The strong generators of $P_{g,k+1}$ are

$\begin{cases} \{g+k+1,g+k+2,\dots,g+2k\} & \mbox{if } 2\le k\le\lceil g/2\rceil, \\
\{g+k+1,g+k+2,\dots,\widehat{2g+2},\dots,g+2k\} & \mbox{if } \lceil g/2\rceil<k\le g, \\
\{g+k+1,g+k+2,\dots,\widehat{d(g+1)},\dots,2g+k+1\} & \mbox{if } k> g.
\end{cases}$
\end{lemma}

Note that in the three cases above, the number of strong generators of $P_{g,k+1}$ is $k$, $k-1$, and $g$, respectively.

\begin{proof}
If $2\le k\le g$, then $P_{g,k+1}=\langle g+1|g+k+1,g+k+2,\dots,\widehat{2g+2},\dots,2g+k+1 \rangle$, and we have $\lambda_1=g+1$, $\lambda_2=g+k+1$.
The strong generators are the elements $g+j$ with $k+1\le j\le 2k$ (with the exception of $2g+2$ in the case that $k>\lceil g/2\rceil$).
Indeed, if $\mu=g+j$ with $k+1\le j\le 2k$, then $\lambda_1+\mu=2g+j+1$ is a minimal generator of $P_{g,k+1}\setminus\{\mu\}$, since
in order to write $2g+j+1$ as a sum of two positive integers, one would have to be strictly less than $\lambda_2=g+k+1$.
On the other hand, if $2k+1\le j\le g+k+1$, then $\lambda_1+\mu=2g+j+1=(g+k+1)+(g+j-k)$, but $g+k+1,g+j-k\in P_{g,k+1}\setminus\{\mu\}$,
so $\lambda_1+\mu$ is not a minimal generator.

A similar argument shows that if $k>g$ then all the effective generators of $P_{g,k+1}$ are strong.
\end{proof}

To improve the bound from Proposition~\ref{prop:lowerbound1}, instead of the labels $\wt{(g)}$ used in $\A'$, we will create a special label $\wt{(g)}_k$ for each semigroup $P_{g,k+1}$ in order to keep track of the number of strong generators.
We will also use the following result that relates strong generators of a numerical semigroup with strong generators of its children in $\T$.

\begin{lemma}\label{lem:gostrong} Let $\L$ be a non-ordinary semigroup, let $\lambda<\mu$ be effective generators, and assume that $\mu$ is a strong generator of $\L$.
Then $\mu$ is a strong generator of $\L\setminus\{\lambda\}$.
\end{lemma}

\begin{proof}
Let $\lambda_1$ be the multiplicity of $\L$. Since $\L$ is not ordinary, $\lambda\neq\lambda_1$, so $\lambda_1$ is also the multiplicity of the semigroup $\L\setminus\{\lambda\}$.
Now since $\mu+\lambda_1$ is a minimal generator of $\L$, it must be a minimal generator of $\L\setminus\{\lambda\}$ as well.
\end{proof}

\begin{theorem}
For $g\ge1$, we have $n_g\ge a_g$, where
$$\sum_{g\ge1} a_g\, t^g=\frac{t\,(1-t^2-2t^3-3t^4+t^5+2t^6+3t^7+3t^8+t^9)}{(1+t)(1-t)(1-t-t^2)(1-t-t^3)(1-t^3-2t^4-2t^5-t^6)}.$$
\label{thm:lowerbound}
\end{theorem}

The first few values of $a_g$ are given in Table~\ref{tab:bounds}.

\begin{proof}
We will construct a generating tree $\A''$ with $\A''\prec\T$ and then count the number of nodes in $\A''$ at each level.
Two kinds of nodes in $\A''$ will correspond directly to nodes in $\T$: a node labeled $\ol{(g+1)}$ for each ordinary semigroup $O_g$, and a node labeled $\wt{(g)}_k$ for each semigroup $P_{g,k+1}$.
The remaining nodes of $\A''$ will be labeled with a pair $(e,s)$, where $e$ and $s$ will be lower bounds on the number of effective and strong generators, respectively, of the corresponding semigroups in $\T$.

It is easy to check that for $3\le i\le g+1$, the child $\langle g+1,g+2,\dots,g+i-1|g+i+1,g+i+2,\dots,2g+1 \rangle$ of $O_g$ has no strong generators.
On the other hand, recall that the genus of $P_{g,k+1}$ is $g+k-1$, and that its number of strong generators is \beq\label{eq:sgk}s(g,k):=\begin{cases} k & \mbox{if } 2\le k\le\lceil g/2\rceil, \\
k-1 & \mbox{if } \lceil g/2\rceil<k\le g, \\
g & \mbox{if } k>g, \end{cases}\eeq
by Lemma~\ref{lem:Pgi}.

The key fact needed in the construction of $\A''$ is the following observation, which is a consequence of Lemmas~\ref{lem:Maria}(b) and~\ref{lem:gostrong}.
For a non-ordinary semigroup $\L=\langle \mu_1,\dots,\mu_{r}|\mu_{r+1},\dots,\mu_{r+e} \rangle$
in which $\mu_{r+1},\mu_{r+2},\dots,\mu_{r+s}$ are strong generators, each child $\L\setminus\{\mu_{r+i}\}=\langle \mu_1,\dots,\mu_{r+i-1}|\mu_{r+i+1},\dots,\mu_{r+e},\mu_1+\mu_{r+i} \rangle$ with $1\le i\le s$
has $e-i+1$ effective generators and at least $s-i$ strong generators $\mu_{r+i+1},\dots,\mu_{r+s}$, while each child $\L\setminus\{\mu_{r+i}\}$ with $s< i\le e$
has either $e-i$ or $e-i+1$ effective generators (depending on whether $\mu_{r+i}$ is strong).

Let $\A''$ be the generating tree with root $\ol{(2)}$ and succession rules
\bea\ol{(e)}&\longrightarrow&(0,0)(1,0)\dots(e-3,0)\wt{(e-1)_2}\ol{(e+1)},\nn\\
\wt{(e)}_k&\longrightarrow&(0,0)(1,0)\dots(e-s-1,0)(e-s+1,0)(e-s+2,1)\dots(e-1,s-2)\wt{(e)}_{k+1},\label{eq:genrulea''}\\
&& \mbox{where }s=s(e,k),\nn\\
(e,s)&\longrightarrow&(0,0)(1,0)\dots(e-s-1,0)(e-s+1,0)(e-s+2,1)\dots(e,s-1).\nn\eea
The first five levels of $\A''$ are shown in Figure~\ref{fig:A''}.

\begin{figure}[htb]
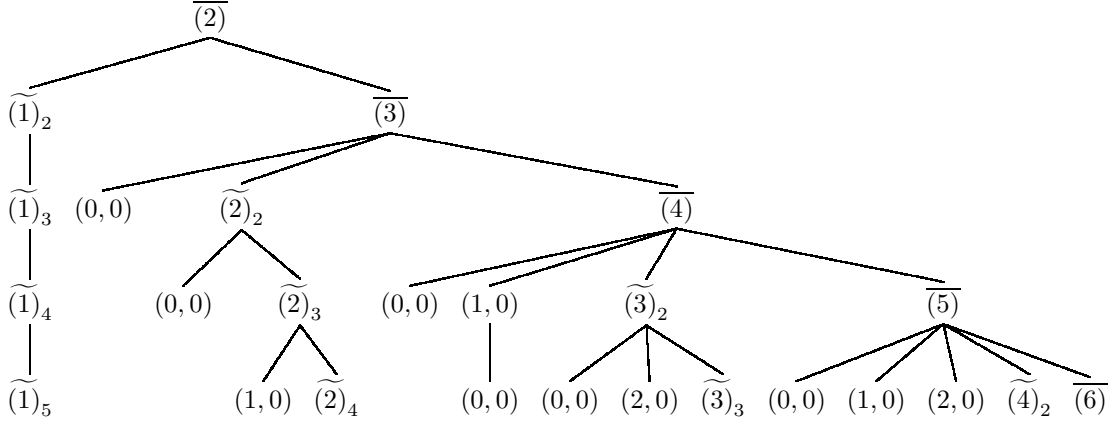

\begin{center}
\childattachsep{0.2in}\childsidesep{0.8em}\synttree [$\ol{(2)}$ [$\wt{(1)}_2$ [ $\wt{(1)}_3$ [$\wt{(1)}_4$ [$\wt{(1)}_5$]]]]
[$\ol{(3)}$ [$(0,0)$] [$\wt{(2)}_2$ [$(0,0)$] [$\wt{(2)}_3$ [$(1,0)$] [$\wt{(2)}_4$]] ]
[$\ol{(4)}$ [$(0,0)$] [$(1,0)$ [$(0,0)$]] [$\wt{(3)}_2$ [$(0,0)$] [$(2,0)$] [$\wt{(3)}_3$]]
[$\ol{(5)}$ [$(0,0)$] [$(1,0)$] [$(2,0)$] [$\wt{(4)}_2$] [$\ol{(6)}$ ]]]]]
\end{center}
\caption{The first five levels of the generating tree $\A''$. \label{fig:A''}}
\end{figure}

The above observation shows that $\mathcal{A}''\prec\T$, since one can recursively construct an embedding of $\A''$ into $\T$
such that each $\ol{(e)}$ is mapped to $O_{e-1}$, each $\wt{(e)_k}$ is mapped to $P_{e,k+1}$, and every node $(e,s)$ in $\A''$ is mapped to a semigroup in $\T$ with at least $e$ effective generators and at least $s$ strong generators.

If we let $a_g$ be the number of nodes of $\A''$ at level $g$, then $n_g\ge a_g$. Next we find an expression for the generating function $\sum_{g\ge1}a_gt^g$.
Let $\ol{F}(u,t)=\sum_{g\ge1} u^{g+1}t^g=\frac{u^2t}{1-ut}$ be again the generating function for ordinary semigroups where $u$ marks the number of effective generators, and let
$$H(u,v,t)=\sum_{g\ge1,k\ge2} u^{g}v^{s(g,k)}t^{g+k-1}$$
be the generating function for the semigroups $P_{g,k+1}$, where the variables $u$, $v$, and $t$ mark the number of effective generators, the number of strong generators, and the genus, respectively.
For fixed $g\ge2$, letting $\gamma=\lceil g/2\rceil$, the coefficient of $u^g$ in $H(u,v,t)$ is
\begin{multline*}
H_g(v,t)=\sum_{k\ge2} v^{s(g,k)}t^{g+k-1}\\
=(v^2t^{g+1}+v^3t^{g+2}+\dots+v^{\gamma}t^{g+\gamma-1})+(v^{\gamma}t^{g+\gamma}+v^{\gamma+1}t^{g+\gamma+1}+\dots+v^{g-1}t^{2g-1})+v^gt^{2g}+v^gt^{2g+1}+\dots\\
=\frac{v^{g}t^{2g}-v^2t^{g+1}+v^{\gamma}t^{g+\gamma}(v-1)}{vt-1}+\frac{v^gt^{2g}}{1-t},
\end{multline*}
by equation~(\ref{eq:sgk}). Including also the semigroups $P_{1,k+1}$, we have
$$H(u,v,t)=\frac{uvt^2}{1-t}+\sum_{g\ge2} H_g(v,t) u^g=
\frac{uvt^2[1+(u^2(v-1)-u)t^2+u^2(1-2v)t^3+u^3v(1-v)t^4+u^3v^2t^5]}{(1-ut)(1-u^2vt^3)(1-t)(1-uvt^2)}.$$
Now let $G(u,v,t)$ be the generating function where the coefficient of $u^ev^st^g$ is the number of nodes in $\A''$ at level $g$ with label $(e,s)$. To get an equation for $G$ in terms of $\ol{F}$, $G$ and $H$,
we use the succession rules~(\ref{eq:genrulea''}) to express the coefficient of $t^{g+1}$ in $G(u,v,t)$ in terms of the coefficients of $t^g$ in the three generating functions.
From the first succession rule we see that each term $u^et^g$ in $\ol{F}(u,t)$ contributes as
$1+u+\dots+u^{e-3}=(u^{e-2}-1)/(u-1)$ to the coefficient of $t^{g+1}$ in $G(u,v,t)$. The third succession rule shows that each term $u^ev^st^g$ in $G(u,v,t)$ contributes to the coefficient of $t^{g+1}$ as
$$1+u+\dots+u^{e-s-1}+u^{e-s+1}+u^{e-s+2}v+\dots+u^{e}v^{s-1}=\frac{u^{e-s}-1}{u-1}+\frac{u^{e+1}v^s-u^{e-s+1}}{uv-1}.$$ Similarly, from the second succession rule, each term $u^ev^st^g$ in $H(u,v,t)$ contributes as
$$\frac{u^{e-s}-1}{u-1}+\frac{u^{e}v^{s-1}-u^{e-s+1}}{uv-1}$$ to the coefficient of $t^{g+1}$ in $G(u,v,t)$. Combining the three contributions, we get the following functional equation for $G$.
\begin{multline}\label{eq:G} G(u,v,t)=t\left[\frac{\ol{F}(u,t)/u^2-\ol{F}(1,t)}{u-1}+\frac{G(u,1/u,t)-G(1,1,t)}{u-1}+u\frac{G(u,v,t)-G(u,1/u,t)}{uv-1}\right.\\
\left. +\frac{H(u,1/u,t)-H(1,1,t)}{u-1}+\frac{H(u,v,t)/v-uH(u,1/u,t)}{uv-1}\right],
\end{multline}
where $\ol{F}$ and $H$ are known.
Collecting the terms with $G(u,v,t)$, the kernel $1-ut/(uv-1)$ is canceled by setting $v=\frac{1+ut}{u}$. This leaves an equation involving only $G(u,1/u,t)$ and $G(1,1,t)$, with kernel $t/(u-1)-1$. Setting
$u=t+1$ to cancel the kernel, we obtain
$$G(1,1,t)=\frac{t^3(1-t^2-5t^4-3t^5+2t^6+5t^7+6t^8+4t^9+t^{10})}{(1+t)(1-t)^2(1-t-t^2)(1-t-t^3)(1-t^3-2t^4-2t^5-t^6)}.$$
Finally, our sought generating function is $$\sum_{g\ge1} a_g\, t^g=\ol{F}(1,t)+H(1,1,t)+G(1,1,t).$$
Note that if it were necessary, an expression for $G(u,v,t)$ could easily be found by first recovering $G(u,1/u,t)$ and then substituting back in equation~(\ref{eq:G}).
\end{proof}

It is easy to refine the above proof by keeping track of the multiplicity of the semigroups. Clearly, the only numerical semigroups whose multiplicity $\lambda_1$ is larger than the Frobenius number (and thus an effective generator)
are the ordinary semigroups. It follows that the multiplicity of a semigroup is passed on to its children in $\T$, with the only exception of the child $O_{g+1}$ of $O_g$.
Adding a new variable $w$ that marks the multiplicity, a proof analogous to that of Theorem~\ref{thm:lowerbound} produces the generating function
$$\frac{w^2t[1-wt^2(1+t+t^2-t^3)-w^2t^3(1+t)(1+t+t^3)+w^3t^5(1+t)^2(1+t+t^2)]}{(1-t)[1-wt(1+t)][1-wt^2(1+t+t^2)][1-w^2t^3(1+t)(1+t+t^2)]}$$ whose coefficient of $w^{\lambda_1}t^g$ gives
a lower bound on the number of numerical semigroups of genus $g$ and multiplicity $\lambda_1$.

\begin{table}[htb]
$$\begin{array}{|r|r|r|r|r|r|r|}
\hline
g & 2F_g & F_{g+2}-1 & a_g & n_g & c_g & 1+3\cdot2^{g-3} \\ \hline
1 & & 1& 1 & 1 & 1 & \\
2 & 2& 2& 2 & 2& 2 & \\
3 & 4& 4& 4& 4& 4 & 4\\
4 & 6& 7& 7& 7& 7 & 7\\
5 & 10& 12& 12& 12& 13 & 13\\
6 & 16& 20& 22& 23& 24 & 25\\
7 & 26& 33& 37& 39& 44 & 49\\
8 & 42& 54& 62& 67& 81 & 97\\
9 & 68& 88& 104& 118& 151 & 193\\
10 & 110& 143& 175& 204& 280 & 385\\
11 & 178& 232& 291& 343& 525 & 769\\
12 & 288& 376& 482& 592& 984 & 1537\\
13 & 466& 609& 796& 1001& 1859 & 3073\\
14 & 754& 986& 1315& 1693& 3511 & 6145\\
15 & 1220& 1596& 2166& 2857& 6682 & 12289\\
16 & 1974& 2583& 3559& 4806& 12709 & 24577\\
17 & 3194& 4180& 5838& 8045& 24334 & 49153\\
18 & 5168& 6764& 9569& 13467& 46565 & 98305\\
19 & 8362& 10945& 15665& 22464& 89626 & 196609\\
20 & 13530& 17710& 25612& 37396& 172381 & 393217\\
21 & 21892& 28656& 41831& 62194& 333262 & 786433\\
22 & 35422& 46367& 68270& 103246& 643733 & 1572865\\
23 & 57314& 75024& 111337& 170963& 1249147 & 3145729\\
24 & 92736& 121392& 181438& 282828& 2421592 & 6291457\\
25 & 150050& 196417& 295480& 467224& 4713715 & 12582913\\
26 & 242786& 317810& 480938& 770832& 9165792 & 25165825\\
27 & 392836& 514228& 782408& 1270267& 17888456 & 50331649\\
28 & 635622& 832039& 1272250& 2091030& 34873456 & 100663297\\
29 & 1028458& 1346268& 2067870& 3437839& 68212220 & 201326593\\
30 & 1664080& 2178308& 3359757& 5646773& 133269997 & 402653185\\
31 & 2692538& 3524577& 5456862&  9266788& 261167821 & 805306369\\
32 & 4356618& 5702886& 8860132& 15195070& 511211652 & 1610612737\\
33 & 7049156& 9227464& 14381714& 24896206& 1003436520 & 3221225473\\
34 & 11405774& 14930351& 23338153& 40761087& 1967293902 & 6442450945\\
35 & 18454930& 24157816& 37863301& 66687201& 3866902804 & 12884901889\\
\hline
\end{array}$$
\caption{\label{tab:bounds} The values for $g\le35$ of the new and previously known bounds on the number $n_g$ of numerical semigroups:
$2F_g\le F_{g+2}-1\le a_g\le n_g\le c_g\le 1+3\cdot2^{g-3}$.}
\end{table}

\section{An improved upper bound}\label{sec:upper}

Whereas the key to the lower bounds in the previous section was to keep track of strong generators, in this section we obtain an upper bound on $n_g$ by keeping the number of healthy generators under control.

\begin{lemma}\label{lem:healthy}
Let  $\L=\langle \mu_1,\dots,\mu_{r}|\mu_{r+1},\dots,\mu_{r+e} \rangle$ be a non-ordinary semigroup where the generators $\mu_{r+i}$ are healthy for $1\le i\le h$ and very weak for $h+1\le i\le e$.
Then the number of healthy generators of $\L\setminus\{\mu_{r+i}\}$ is\\
$$\begin{cases} \le\min\{h-i+2,e-i+1\} & \mbox{for }1\le i\le h,\\
\le\min\{1,e-h-1\} & \mbox{for } i=h+1,\\
0 & \mbox{for } i\ge h+2.
\end{cases}$$
\end{lemma}

\begin{proof}
Let $1\le i\le e$, and denote $\L_i=\L\setminus\{\mu_{r+i}\}$.
We know by Lemma~\ref{lem:Maria} that the effective generators of $\L_i$ are $\mu_{r+i+1},\mu_{r+i+2},\dots,\mu_{r+e}$, plus $\mu_1+\mu_{r+i}$ if $\mu_{r+i}$ is a strong, which can only happen for $i\le h$.
Thus, $e-i+1$ (resp., $e-i$) is an upper bound on the number of effective generators if $i\le h$ (resp., $i>h$),
so in particular it is an upper bound on the number of healthy generators. 

Let $g$ be the genus of $\L$. For $1\le j\le e$, the generator $\mu_{r+j}$ is healthy in $\L$ if $\mu_{r+j}\le 2g+3-\mu_1$ by definition. Since the genus of $\L_i$ is $g+1$,
$\mu_{r+j}$ is a healthy generator of $\L_i$ if $j>i$ and $\mu_{r+j}\le 2g+5-\mu_1$.
When $i\le h$, it follows that aside from $\mu_{r+i+1},\mu_{r+i+2},\dots,\mu_{r+h}$, which were already healthy in $\L$, the only possible new healthy generators of $\L_i$
are $\mu_{r+h+1}$, $\mu_{r+h+2}$, and $\mu_1+\mu_{r+i}$. For all three to be healthy in $\L_i$, they would need to satisfy $2g+3-\mu_1<\mu_{r+h+1}<\mu_{r+h+2}<\mu_1+\mu_{r+i}\le 2g+5-\mu_1$.
Thus, $\L_i$ has at most two new healthy generators aside from the $h-i$ generators that were already healthy in $\L$.

For $i=h+1$, neither $\mu_{r+h+1}$ nor $\mu_1+\mu_{r+h+1}$ are generators of $\L_i$, so the only possible healthy generator is $\mu_{r+h+2}$. 
For $i>h+1$, $\L_i$ has no healthy generators.
\end{proof}

\begin{theorem}
For $g\ge1$, we have $n_g\le c_g$, where
$$\sum_{g\ge1} c_g\, t^g=t\,\frac{2-3t+t^2-4t^3+3t^4-2t^5+t(1-t-t^3)\sqrt{(1+2t)/(1-2t)}}{2(1-3t+3t^2-3t^3+4t^4-3t^5+2t^6)}.$$
\label{thm:upperbound}
\end{theorem}

The first few values of $c_g$ are given in Table~\ref{tab:bounds}. Note that this generating function has two singularities at $1/2$ and $-1/2$. Standard singularity analysis techniques from~\cite[Chapter VI]{FS} show that
the coefficients $c_g$ grow asymptotically like $2^g/\sqrt{\pi g}$. This is far from the asymptotic behavior that is implied by equation~(\ref{eq:fibconj}), from where one should expect that $\lim_{g\rightarrow\infty}(n_g)^{1/g}=\phi$.

\begin{proof}
We will construct a generating tree $\C$ with $\T\prec\C$ and then count the number of nodes in $\C$ at each level.
Aside from the nodes $\ol{(g+1)}$ that correspond to ordinary semigroups $O_g$, the other nodes of $\C$ will have a pair of labels $(e,h)$,
where $e$ and $h$ will be upper bounds on the number of effective and strong generators, respectively, of the corresponding semigroups in $\T$.

Let us first look at the number of healthy generators of the non-ordinary children of $O_g$. For $g\ge1$, the child $\langle g+1|g+3,g+4,\dots,2g+1,2g+3 \rangle$ has two healthy generators $g+3$ and $g+4$.
For $g\ge2$ and $3\le i\le g+1$, the child $\langle g+1,g+2,\dots,g+i-1|g+i+1,g+i+2,\dots,2g+1 \rangle$ has one healthy generator $g+i+1$ if $i=3$, and no healthy generators otherwise.

Let $\C$ be the generating tree with root $\ol{(2)}$ and succession rules
\bea\ol{(e)}&\longrightarrow&(0,0)(1,0)\dots(e-4,0)(e-3,\min\{1,e-3\})(e-1,\min\{2,e-1\})\ol{(e+1)},\nn\\
(e,h)&\longrightarrow&(0,0)(1,0)\dots(e-h-2,0)(e-h-1,\min\{1,e-h-1\})\nn\\ 
&&(e-h+1,\min\{2,e-h+1\})(e-h+2,\min\{3,e-h+2\})\dots(e,\min\{h+1,e\}).\nn\eea
From Lemmas~\ref{lem:Maria}(b) and~\ref{lem:healthy} it follows that $\T\prec\C$.
Indeed, an embedding from $\T$ to $\C$ can be given so that each every non-ordinary semigroup in $\T$ with $e'$ effective generators and $h'$ healthy ones is mapped to a
node $(e,h)$ with $e'\le e$ and $h'\le h$.

Letting $c_g$ be the number of nodes at level $g$ in $\C$, we have that $n_g\le c_g$. To find a generating function for the sequence $c_g$,
it will be convenient to relabel each node $(e,h)$ of $\C$
with the pair $(d,h)$, where $d=e-h$. With this new labeling, the above succession rules for $\C$ can be rewritten as

\bea\ol{(e)}&\longrightarrow&\begin{cases}(0,0)(1,0)\dots(e-4,0)(e-4,1)(e-3,2)\ol{(e+1)} & \mbox{if }e\ge4,\\
(0,0)(0,2)\ol{(4)} & \mbox{if }e=3,\\
(0,1)\ol{(3)} & \mbox{if }e=2,
\end{cases}\label{eq:genrulee}\\
(d,h)&\longrightarrow&
\begin{cases}(0,0)(1,0)\dots(d-2,0)(d-2,1)(d-1,2)(d-1,3)\dots(d-1,h+1) & \mbox{if }d\ge 2, \\
(0,0)(0,2)(0,3)\dots(0,h+1) & \mbox{if }d=1, \\
(0,1)(0,2)\dots(0,h) & \mbox{if }d=0. \end{cases}
\label{eq:genruled}\\
\nn\eea
Figure~\ref{fig:C} shows the first five levels of $\C$ with the new labels.

\begin{figure}[htb]
\begin{center}
\childattachsep{0.2in}\childsidesep{0.8em}\synttree [$\ol{(2)}$ [$(0,1)$ [ $(0,1)$ [$(0,1)$ [$(0,1)$]]]]
[$\ol{(3)}$ [$(0,0)$] [$(0,2)$ [$(0,1)$ [$(0,1)$]] [$(0,2)$ [$(0,1)$] [$(0,2)$]] ]
[$\ol{(4)}$ [$(0,0)$] [$(0,1)$ [$(0,1)$]] [$(1,2)$ [$(0,0)$] [$(0,2)$] [$(0,3)$]]
[$\ol{(5)}$ [$(0,0)$] [$(1,0)$] [$(1,1)$] [$(2,2)$] [$\ol{(6)}$ ]]]]]
\end{center}
\caption{The first five levels of the generating tree $\C$. \label{fig:C}}
\end{figure}

Let $K(x,v,t)$ be the generating function where the coefficient of $x^dv^ht^g$ is the number of nodes in $\C$ at level $g$ with label $(d,h)$, and let
$K_d(v,t)=[x^d]K(x,v,t)$, that is, $K_d(v,t)$ is a generating function for the nodes $(d,h)$ with fixed $d\ge0$. Clearly,
$$\sum_{g\ge1} c_g\, t^g=K(1,1,t)+\frac{t}{1-t}$$ counts the total number of nodes in $\C$ at each level,
since $\ol{F}(1,t)=\frac{t}{1-t}$ is the generating function for the nodes with labels of the form $\ol{(e)}$.

We now use the succession rules~(\ref{eq:genrulee}) and~(\ref{eq:genruled}) to get an equation for $K_d(v,t)$ for $d\ge1$. Since the second and third cases of rule~(\ref{eq:genruled}) yield only labels of the form $(0,*)$, we
just need to look at the first case of rule~(\ref{eq:genruled}). We see that the coefficient of $t^{g+1}$ in $K_d(v,t)$ gets a contribution of $v^2+v^3+\dots+v^{h+1}=v^2(v^h-1)/(v-1)$ from each term $v^ht^g$ in $K_{d+1}(v,t)$, plus a
contribution of $v$ from each $v^ht^g$ in $K_{d+2}(v,t)$, and a contribution of $v^0$ from each $v^ht^g$ in $K_{i}(v,t)$ for every $i\ge d+2$.
Similarly, the first case of rule~(\ref{eq:genrulee}) shows that the coefficient of $t^{g+1}$ in $K_d(v,t)$ gets a contribution of $v^2$ from each node labeled $\ol{(d+3)}$ at level $g$ (there is exactly one such node when $g=d+2$),
a contribution of $v$ from each node labeled $\ol{(d+4)}$ at level $g$ (this happens when $g=d+3$), and a contribution of $v^0$ from each node at level $g$ labeled $\ol{(e)}$ with $e\ge d+4$ (there is one such node for each $g\ge d+3$).
Putting this together, we get the following functional equation for $K_d(v,t)$ with $d\ge1$:
\beq\label{eq:Kd} K_d(v,t)=t\left[\frac{v^2}{v-1}\left(K_{d+1}(v,t)-K_{d+1}(1,t)\right)+v K_{d+2}(1,t)+\sum_{i\ge d+2}K_i(1,t)+v^2t^{d+2}+vt^{d+3}+\frac{t^{d+3}}{1-t}\right]\eeq

Instead of an equation for each $K_d$, it would have been natural to seek a functional equation for $K(x,v,t)$ (or for $\sum_{d\ge1} K_d(v,t)x^d$, since $K_0(v,t)$ can be recovered at the end).
However, such functional equation obtained from the succession rules~(\ref{eq:genrulee}) and~(\ref{eq:genruled}) by the standard method also involves the unknown individual functions $K_1(v,t)$ and $K_2(v,t)$,
and it cannot be solved by applying the kernel method in the usual way. The cause of this problem is the fact that the right hand side of succession rule~(\ref{eq:genruled}) depends on the value of $d$.
(Note that the dependance of rule~(\ref{eq:genrulee}) on the value of $e$ does not cause trouble because we have control of where the nodes $\ol{(e)}$ with $e=2,3$ appear in the tree.)

They key to solving equation~(\ref{eq:Kd}) and overcoming this problem is to realize the following fact.

{\bf Claim.} For $d\ge2$, $K_d(v,t)=t\,K_{d-1}(v,t)$.

To prove this, we will show that $[t^g]K_d(v,t)=[t^{g-1}]K_{d-1}(v,t)$ by induction on $g$. It is easy to check that for $g=1$, $[t^g]K_d(v,t)=0=[t^{g-1}]K_{d-1}(v,t)$. Now, given $g\ge 2$,
from equation~(\ref{eq:Kd}) we have
\begin{multline}[t^g]K_d(v,t)=\frac{v^2}{v-1}\left([t^{g-1}]K_{d+1}(v,t)-[t^{g-1}]K_{d+1}(1,t)\right)+v [t^{g-1}]K_{d+2}(1,t)+\sum_{i\ge d+2}[t^{g-1}]K_i(1,t)\\
+v^2\chi_{g=d+3}+v\chi_{g=d+4}+\chi_{g\ge d+4},\nn\end{multline}
where $\chi_E$ is the indicator variable for the event $E$. By the induction hypothesis, the right hand side of the above equation equals
\begin{multline}\frac{v^2}{v-1}\left([t^{g-2}]K_{d}(v,t)-[t^{g-2}]K_{d}(1,t)\right)+v [t^{g-2}]K_{d+1}(1,t)+\sum_{i\ge d+1}[t^{g-2}]K_i(1,t)\\
+v^2\chi_{g=d+3}+v\chi_{g=d+4}+\chi_{g\ge d+4},\nn\end{multline}
which equals $[t^{g-1}]K_{d-1}(v,t)$ again by equation~(\ref{eq:Kd}), thus proving the claim.

The claim implies that for $d\ge2$, $K_d(v,t)=t^{d-1} K_{1}(v,t)$. Plugging this into equation~(\ref{eq:Kd}) with $d=1$ yields an equation involving only $K_1$:
\beq\label{eq:K1} K_1(v,t)=t\left[\frac{v^2 t}{v-1}\left(K_{1}(v,t)-K_{1}(1,t)\right)+v t^2 K_{1}(1,t)+\frac{t^2}{1-t}K_1(1,t)+v^2t^{3}+vt^{4}+\frac{t^4}{1-t}\right].\eeq
Collecting the terms with $K_1(v,t)$ we see that the kernel of the equation is $1-v^2 t^2/(v-1)$, which is canceled with the substitution $v=\frac{1-\sqrt{1-4t^2}}{2t^2}$. Solving for $K_1(1,t)$ we get
\beq\label{eq:K1sol} K_1(1,t)=\frac{1-2t-t^2+3t^3-t^4+4t^5-8t^6+6t^7-4t^8-(1-2t+t^2-t^3+t^4)\sqrt{1-4t^2}}{2(1-3t+3t^2-3t^3+4t^4-3t^5+2t^6)}.\eeq

Now, to find $K_0$, we use again the succession rules~(\ref{eq:genrulee}) and~(\ref{eq:genruled}) to get a functional equation for it.
Looking at rule~(\ref{eq:genruled}) for $d=0$, we see that the coefficient of $t^{g+1}$ in $K_0(v,t)$ gets a contribution of $v+v^2+\dots+v^{h}=v(v^h-1)/(v-1)$ from each term $v^ht^g$ in $K_{0}(v,t)$. The same
rule for $d=1$ shows a contribution of $v^0+v^2+v^3+\dots+v^{h+1}=1+v^2(v^h-1)/(v-1)$ from each $v^ht^g$ in $K_{1}(v,t)$. The rule for $d\ge 2$ reveals a contribution of $v^0$ from each $v^ht^g$ in $K_{d}(v,t)$ with $d\ge 2$,
plus a contribution of $v$ from each $v^ht^g$ in $K_{2}(v,t)$.
Similarly, rule~(\ref{eq:genrulee}) shows that the coefficient of $t^{g+1}$ in $K_0(v,t)$ gets a contribution of $v$ from the node labeled $\ol{(2)}$ when $g=1$,
a contribution of $v^0+v^2$ from the node labeled $\ol{(3)}$ when $g=2$, a contribution of $v^0$ from each node at level $g$ labeled $\ol{(e)}$ with $e\ge 4$ (there is one such node for each $g\ge 3$),
and a contribution of $v$ from the node labeled $\ol{(4)}$ when $g=3$.
All these contributions yield the following equation for $K_0(v,t)$:
\begin{multline}\nn K_0(v,t)=t\left[\frac{v}{v-1}\left(K_{0}(v,t)-K_{0}(1,t)\right)+\frac{v^2}{v-1}\left(K_{1}(v,t)-K_{1}(1,t)\right)+K_1(1,t)\right.\\
\left.+\sum_{d\ge2}K_d(1,t)+vK_2(1,t)+vt+(1+v^2)t^2+\frac{t^3}{1-t}+vt^3\right].\end{multline}
Using equation~(\ref{eq:K1sol}) and the fact that $K_d(v,t)=t^{d-1} K_{1}(v,t)$ for $d\ge2$, and canceling the kernel with $v=1/(1-t)$, we obtain
\beq\label{eq:K0sol}K_0(1,t)=\frac{-1+t+3t^2-2t^3-5t^5+4t^6-4t^7+(1-3t+3t^2-2t^3+2t^4-t^5)\sqrt{(1+2t)/(1-2t)}}{2(1-3t+3t^2-3t^3+4t^4-3t^5+2t^6)}.\eeq

Finally, from equations~(\ref{eq:K1sol}) and~(\ref{eq:K0sol}) and the fact that
$$\sum_{g\ge1} c_g\,t^g=K_0(1,t)+\frac{K_1(1,t)}{1-t}+\frac{t}{1-t}$$ we get the stated generating function.
\end{proof}

\section{Infinite chains}\label{sec:infinite}

An infinite sequence of numerical semigroups $\L_0=\N_0,\L_1,\L_2,\dots$ is called an {\em infinite chain} if each $\L_i$ is the parent of $\L_{i+1}$ in $\T$. Clearly, numerical semigroup $\L$ belongs to some infinite
chain if and only if it has an infinite number of descendants in $\T$. Let $\T^\infty$ be the subtree of $\T$ consisting of the semigroups with an infinite number of descendants in $\T$.
With some abuse of notation, we write $\L\in\T^\infty$ if $\L$ is node in $\T^\infty$.
Let $m_g$ be the number of nodes of $\T^\infty$ at level $g$, that is, the number of semigroups of genus $g$ with an infinite number of descendants in $\T$.
In this section we give upper bounds on $m_g$. First we obtain a Fibonacci-like upper bound using an argument based on the work from the previous sections, together with the following observation.

\begin{lemma}
Let $\L$ be a numerical semigroup with multiplicity $\lambda_1$ and $e$ effective generators. If $\lambda_1>2e$, then $\L\notin\T^\infty$.
\end{lemma}

\begin{proof}
First of all, note that the condition $\lambda_1>2e$ forces $\L$ to be non-ordinary.
Let $g$ be the genus of $\L$. Recall that every effective generator $\mu$ of $\L$ must satisfy $\mu\le 2g+1$.
We fix $\lambda_1$ and proceed by induction on $e$. The result clearly holds when $e=0$, since in that case $\L$ has no children.

Let now $\L$ be a numerical semigroup with $e>1$ effective generators $\nu_1<\dots<\nu_e$ and multiplicity $\lambda_1>2e$,
and assume that no semigroup with multiplicity $\lambda_1$ and less than $e$ effective generators belongs to $\T^\infty$.
Suppose for contradiction that $\L\in\T^\infty$. All the children of $\L$ have less than $e$ effective generators except possibly the child $\L_1:=\L\setminus\{\nu_1\}$. Thus, $\L\in\T^\infty$ implies that
$\L_1\in\T^\infty$, and the effective generators of $\L_1$ are $\nu_2<\dots<\nu_e<\lambda_1+\nu_1$. By the same argument, one of the children of $\L_1$ must have infinitely many descendants, which forces the child $\L_2:=\L_1\setminus\{\nu_2\}$ to have effective generators $\nu_3<\dots<\nu_e<\lambda_1+\nu_1<\lambda_1+\nu_2$. After $e$ steps, we get a child $\L_e$ of genus $g+e$ with effective generators
$\lambda_1+\nu_1<\dots<\lambda_1+\nu_e$. 
After $ke$ steps we get a semigroup $\L_{ke}$ of genus $g+ke$ whose smallest effective generator is $k\lambda_1+\nu_1$.
For sufficiently large $k$ (take $k>(2g+1-\nu_1)/(\lambda_1-2e)$), we have that $k\lambda_1+\nu_1>2(g+ke)+1$, which contradicts the fact
that $k\lambda_1+\nu_1$ is an effective generator.
\end{proof}

\begin{prop}\label{prop:boundmgfib}
For $g\ge4$, the number $m_g$ of numerical semigroups of genus $g$ with an infinite number of descendants in $\T$
satisfies $$m_g\le 2F_{g-1}.$$
\end{prop}

\begin{proof}
Let $\mathcal{B}$ be the tree described at the end of Section~\ref{sec:intro}. We now define a tree $\B'$ isomorphic to $\B$ by adding  to the labels additional information about the multiplicity of the corresponding semigroups. Recall that the multiplicity of a non-ordinary semigroup equals the multiplicity of its parent.
Let the root of $\B'$ be again $\ol{(2)}$, and let its succession rules be
\bean\ol{(\lambda)}&\longrightarrow&(0,\lambda)(1,\lambda)\dots(\lambda-3,\lambda)(\lambda-1,\lambda)\ol{(\lambda+1)},\\ (e,\lambda)&\longrightarrow&(1,\lambda)(2,\lambda)\dots(e,\lambda).\eean
In the same way that $\T$ can be embedded in $\mathcal{B}$, there is an embedding of $\T$ into $\B'$
mapping each ordinary semigroup $O_g$ to $\ol{(g+1)}$ and each non-ordinary semigroup in $\T$ with $e$ effective generators and multiplicity $\lambda$
to a node $(e',\lambda)$ in $\B'$ with $e<e'$. Restricting this map to $\T^\infty$, the nodes $(e,\lambda)$ with $2e<\lambda$ are no longer needed,
so we get an embedding of $\T^\infty$ into the tree $\I$ with root $\ol{(2)}$ and succession rules
\bea\nn \ol{(\lambda)}&\longrightarrow&(\lceil \lambda/2 \rceil,\lambda)(\lceil \lambda/2 \rceil+1,\lambda)\dots(\lambda-3,\lambda)(\lambda-1,\lambda)\ol{(\lambda+1)},\\ (e,\lambda)&\longrightarrow&(\lceil \lambda/2 \rceil,\lambda)(\lceil \lambda/2 \rceil+1,\lambda)\dots(e,\lambda).\label{eq:genrulei}\eea
The first six levels of $\I$ are drawn in Figure~\ref{fig:I}.

\begin{figure}[htb]
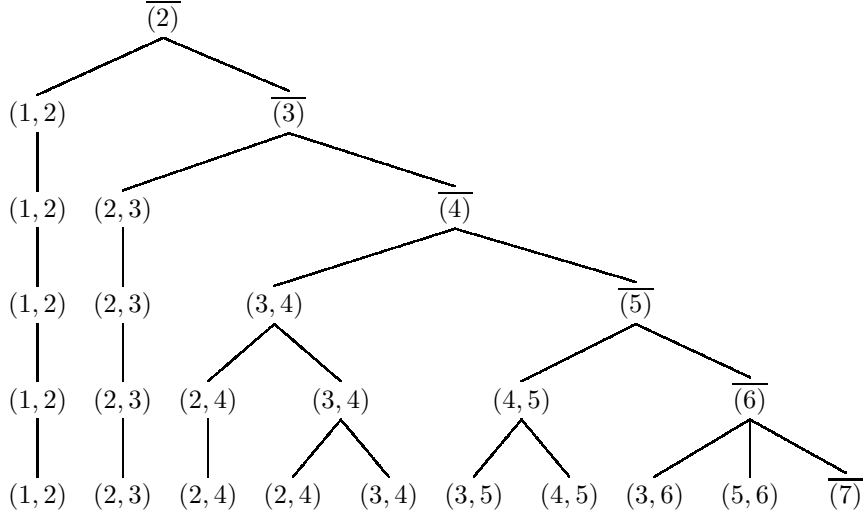

\begin{center}
\synttree [$\ol{(2)}$ [$(1,2)$ [ $(1,2)$ [$(1,2)$ [$(1,2)$ [$(1,2)$]]]]]
[$\ol{(3)}$ [$(2,3)$ [$(2,3)$ [$(2,3)$ [$(2,3)$]]] ]
[$\ol{(4)}$ [$(3,4)$ [$(2,4)$ [$(2,4)$]] [$(3,4)$ [$(2,4)$] [$(3,4)$]]]
[$\ol{(5)}$ [$(4,5)$ [$(3,5)$] [$(4,5)$]] [$\ol{(6)}$ [$(3,6)$] [$(5,6)$] [$\ol{(7)}$] ]]]]]
\end{center}
\caption{The first six levels of the generating tree $\I$. \label{fig:I}}
\end{figure}

Let $d_g$ be the number of nodes of $\I$ at level $g$. Since $\T^\infty<\I$, we have that $m_g\le d_g$. To find the generating function for the sequence $d_g$, let
$\ol{F}(u,t)$ be defined as before, and let $J(u,v,t)$ be the generating function where the coefficient of $u^e v^\lambda t^g$ is the number of nodes in $\I$ at level $g$ with label $(e,\lambda)$.
In order to translate the rules~(\ref{eq:genrulei}) into functional equations, it will be convenient to separate the terms in $J(u,v,t)$ according to the parity of the exponent of $v$, so that
$J(u,v,t)=J_e(u,v,t)+J_o(u,v,t)$, with $e$ and $o$ standing for {\it even} and {\it odd}. Also, let $\ol{F}_e(u,t)=\frac{u^4t^3}{1-u^2t^2}$ and $\ol{F}_o(u,t)=\frac{u^5t^4}{1-u^2t^2}$, so that
$\ol{F}(u,t)=u^2t+u^3t^2+\ol{F}_e(u,t)+\ol{F}_o(u,t)$.
The coefficient of $t^{g+1}$ in $J_e(u,v,t)$ gets a contribution of $u^{\lceil \lambda/2 \rceil}v^\lambda+\dots+u^ev^\lambda=\frac{u^{e+1}-u^{\lceil \lambda/2 \rceil}}{u-1}v^\l$ from each term $u^ev^\l t^g$ in $J_e(u,v,t)$, and a contribution of
$u^{\lceil \lambda/2 \rceil}v^\lambda+\dots+u^{\lambda-3}v^\lambda+u^{\lambda-1}v^\lambda=\frac{u^{\l-2}-u^{\lceil \lambda/2 \rceil}}{u-1}v^\l+u^{\l-1}v^\l$ from each term $u^\l t^g$ in $\ol{F}_e(u,t)$. This does not include the term $uv^2t^2$ coming from the first rule when $\l=2$. In terms of the generating functions,
$$J_e(u,v,t)=uv^2t^2+t\left[\frac{uJ_e(u,v,t)-J_e(1,\sqrt{u}v,t)}{u-1}+\frac{\ol{F}_e(uv,t)/u^2-\ol{F}_e(\sqrt{u}v,t)}{u-1}+\ol{F}_e(uv,t)/u\right].$$
Defining a new variable $w=uv^2$ and letting $\widehat{J}_e(u,w,t)=J_e(u,\sqrt{w/u},t)=J_e(u,v,z)$, the equation becomes
\beq\label{eq:Ke}\widehat{J}_e(u,w,t)=wt^2+t\left[\frac{u\widehat{J}_e(u,w,t)-\widehat{J}_e(1,w,t)}{u-1}+\frac{w^2t^3(wt^2(1-u)+u)}{(1-uwt^2)(1-wt^2)}\right].\eeq
The kernel is canceled setting $u=1/(1-t)$. Solving the resulting equation for $\widehat{J}_e(1,w,t)$ and substituting back in~(\ref{eq:Ke}) we
get $$\widehat{J}_e(u,w,t)=\frac{wt^2(1-wt^2+w^2t^4)}{(1-uwt^2)(1-t-wt^2)}.$$
An expression for $\widehat{J}_o(u,w,t)=J_o(u,v,t)$ can be obtain analogously. Finally, we get
$$\sum_{g\ge1} d_g\,t^g=\widehat{J}_e(1,1,t)+\widehat{J}_o(1,1,t)+\ol{F}(1,t)=\frac{t(1+t-t^3-t^4)}{1-t-t^2}=t^3+t+\frac{2t^2}{1-t-t^2}=t+t^3+\sum_{g\ge1}2F_{g-1}\,t^g.$$
Note that for $1\le g\le4$, $d_g=g=m_g$.
\end{proof}

The above bound can be significantly improved if we use the following result from~\cite{BB}, which characterizes semigroups with an infinite number of descendants in terms of the greatest common divisor of the elements
smaller than the Frobenius number.

\begin{theorem}[\cite{BB}]\label{thm:gcd}
Let $\L$ be a numerical semigroup with genus $g$ and Frobenius number $f$.
Then, $\L\in\T^\infty$ if and only if $\gcd(\lambda_0,\dots,\lambda_{f-g})\neq1$.
\end{theorem}

This allows us to compare the number of semigroups in infinite chains with the total number of semigroups of each genus.

\begin{prop}\label{prop:boundmg} For $g\ge1$,
$$m_g\le 1+(g-1)\sum_{i=0}^{\lfloor (g-1)/2 \rfloor} n_i.$$
\end{prop}

\begin{proof}
Given a non-ordinary numerical semigroup $\L\in\T^\infty$ with genus $g$ and Frobenius number $f$, let $d=\gcd(\lambda_0,\lambda_1,\dots,\lambda_{f-g})$.
We know from Theorem~\ref{thm:gcd} that $d\neq1$. Let $\wt{\lambda_i}=\lambda_i/d$ for $0\le i\le f-g$,
let $\ell=\lfloor f/d \rfloor$, and let
$$\wt\L=\{\wt\lambda_0,\wt\lambda_1,\dots,\wt\lambda_{f-g},\ell+1,\ell+2,\dots\}.$$
Then $\wt\L$ is a numerical semigroup of genus $\wt g=g+\ell-f$. Note that since $d\ge2$, we have $\ell\le f/2$,
so $$\wt g\le g+\frac{f}{2}-f=g-\frac{f}{2}\le\frac{g-1}{2},$$
where in the last inequality we use that $f\ge g+1$ for non-ordinary semigroups.

Denoting $g'=\lfloor (g-1)/2 \rfloor$, this defines a map $$\begin{array}{rcl}
\T^\infty_g\setminus\{O_g\}& \longrightarrow& \T_{\le g'}\times\{2,3,\dots,g\}\\
\L&\mapsto&(\wt\L,d)\end{array}$$
where $\T^\infty_g\setminus\{O_g\}$ is the set of non-ordinary semigroups of genus $g$ in infinite chains,
and $\T_{\le g'}$ is the set of numerical semigroups of genus at most $g'$, including the semigroup of genus $0$.
Furthermore, this map is injective because given $\wt\L$ and $d$, we can recover $\L$ by multiplying the elements of $\wt\L$ by $d$,
and adding all the missing integers greater than the $g$-th gap.
Counting cardinalities it follows that
$$m_g-1\le (g-1)\sum_{i=0}^{g'} n_i.$$
\end{proof}

Even though for small values of $g$ the bound from Proposition~\ref{prop:boundmgfib} is smaller, Proposition~\ref{prop:boundmg} gives a better asymptotic bound using the fact
that $\lim_{g\rightarrow\infty}(n_g)^{1/g}\le 2$.

\begin{corollary} We have
$$\lim_{g\rightarrow\infty}(m_g)^{1/g}\le \sqrt{2}.$$
\end{corollary}

\begin{proof}
Denoting again $g'=\lfloor (g-1)/2 \rfloor$ and
using that $n_i\le 1+3\cdot2^{i-3}$ for $i\ge3$ (see equation~(\ref{eq:brasbounds})), Proposition~\ref{prop:boundmg} implies that
$$m_g\le 1+(g-1)(2+g'+3\sum_{j=0}^{g'-3}2^j)=1+(g-1)(g'+3\cdot2^{g'-2}-1).$$
The result follows now taking limits.
\end{proof}

In fact, if the conjectured equation~(\ref{eq:fibconj}) holds, Proposition~\ref{prop:boundmg} implies that
$$\lim_{g\rightarrow\infty}(m_g)^{1/g}\le \sqrt{\phi}\approx 1.27201965.$$

\subsection*{Acknowledgement}
The author is grateful to Maria Bras-Amor\'os for helpful comments and suggestions.

\end{document}